\documentclass{article}
\usepackage{amsmath,amssymb,enumerate,ascmac,mathrsfs,amsthm}
\usepackage{hyperref}

\newtheorem{definition}{Definition}[section]
\newtheorem{theorem}[definition]{Theorem}
\newtheorem{proposition}[definition]{Proposition}
\newtheorem{lemma}[definition]{Lemma}
\newtheorem{remark}[definition]{Remark}
\newtheorem{corollary}[definition]{Corollary}
\newtheorem{example}[definition]{Example}

\newcommand{\bP}{\mathbb{P}}
\newcommand{\bR}{\mathbb{R}}
\newcommand{\bE}{\mathbb{E}}
\newcommand{\cF}{\mathcal{F}}
\newcommand{\cG}{\mathcal{G}}
\newcommand{\cL}{\mathcal{L}}
\newcommand{\cS}{\mathcal{S}}
\newcommand{\al}{\alpha}
\newcommand{\be}{\beta}
\newcommand{\de}{\delta}
\newcommand{\ep}{\varepsilon}

\newcommand{\la}{\lambda}

\newcommand{\Ga}{\Gamma}
\newcommand{\sgn}{\mathrm{sgn}}

\numberwithin{equation}{section}

\allowdisplaybreaks

\begin{document}

\frenchspacing

\title{\bf A potential theoretic approach to Tanaka formula for asymmetric L\'evy processes}
\author{Hiroshi TSUKADA
\footnote{Graduate School of Science,
Osaka City University,
Japan.
e-mail: \texttt{hrstsukada@gmail.com}}}
\date{}
\maketitle

\begin{abstract}
In this paper, we shall introduce the Tanaka formula from viewpoint of the Doob--Meyer decomposition.
For symmetric L\'evy processes, if the local time exists, Salminen and Yor \cite{Sal} obtained the Tanaka formula by using the potential theoretic techniques.
On the other hand, for asymmetric stable processes with index $\al \in (1,2)$, we studied in \cite{Tsu} by using It\^o's stochastic calculus and the Fourier analysis.
In this paper, we study the Tanaka formula for asymmetric L\'evy processes via the potential theoretic approach.
We give several examples for important processes.
Our approach also gives the invariant excessive function with respect to the killed process in the case of asymmetric L\'evy processes and it generalized the result in Yano \cite{Yan2}.
\end{abstract}

\section{Introduction}\label{sec1}
In this paper,
we shall focus on local times for L\'evy processes.
It is known that there are several definitions of local times for different stochastic processes, see Geman and Horowitz \cite{Gem}.
Thus, we define a local time $L=\{L^x_t:x\in\bR,t\geq0\}$ for a L\'evy process $X$ by
the occupation density which means
random variables $L=\{L^x_t:x\in\bR,t\geq0\}$ satisfying
for each non-negative Borel measurable function $f$ and $t\geq0$,
\begin{equation*}
\int^t_0 f(X_s) ds = \int_\bR f(a) L^a_t da \quad \text{a.s.},
\end{equation*}
and is chosen as
\begin{equation*}
L^x_t := \limsup_{\ep \downarrow 0}\frac{1}{2\ep}\int^t_0 1_{\{|X_s-x|<\ep\}}ds.
\end{equation*}
A local time is an important amount to study the reflection problem (see, e.g. Chung and Williams \cite{Chu}) and the Ray--Knight theorem (see, e.g. Eisenbaum et al. \cite{Eis}).
In the case of Brownian motions, the Tanaka formula is an important expression to understand these problems.
Thus, in the case of L\'evy processes we expect that the Tanaka formula is a useful tool to consider those problems.

For a real-valued Brownian motion $B=(B)_{t\geq0}$, it is well known that the Tanaka formula holds:
\begin{equation*}
|B_t-x| - |B_0-x| = \int^t_0 \sgn(B_s-x) dB_s + L^x_t,
\end{equation*}
where $L^x_t$ denotes the local time of the Brownian motion at level $x$.
It represents that the local time $L^x$ can be understood as a bounded variation process in the Doob--Meyer decomposition on the positive submartingale $|B-x|$.
Our goal in this paper is to construct the Tanaka formula from the viewpoint of the Doob--Meyer decomposition.

The Tanaka formula has already studied for symmetric stable processes with index $\al\in(1,2)$ by Yamada \cite{Yam}, for symmetric L\'evy processes by Salminen and Yor \cite{Sal}.
In this paper, we are interested in asymmetric L\'evy processes, while the formula has been obtained for asymmetric stable processes in \cite{Tsu}.
We shall make the Tanaka formula for asymmetric L\'evy processes based upon the potential approach as stated in \cite{Sal}.
Moreover, it will clearly extend the original Tanaka formula for Brownian motions to our formula for asymmetric L\'evy processes.

In \cite{Tsu}, we have already obtained the Tanaka formula for asymmetric stable processes with index $\al \in (1,2)$ via It\^o's stochastic calculus.
By using the Fourier transform,
we can obtain the fundamental solution $F$ of the infinitesimal generator $\cL$ for asymmetric stable processes:
\[ \cL F(x) = \de_0(x) \]
where $\de_0$ is the Dirac delta function, in the sense of Schwartz distribution.
We can construct the Tanaka formula for an asymmetric stable process $S =(S_t)_{t\geq0}$
with index $\al \in (1,2)$ by using It\^o's stochastic calculus and the scaling property of stable processes:
\[ F(S_t-x)-F(S_0-x)= \tilde{N}^x_t +L^x_t \]
where the process $(\tilde{N}^x_t)_{t\geq0}$ given by
\[ \tilde{N}^x_t :=\int^t_0\int_{\bR\setminus\{0\}}\{F(S_{s-}-x+h)-F(S_{s-}-x)\}
\tilde{N}(ds,dh) \]
is a square integrable martingale and $L^x_t$ is the local time at level $x$.
Here, $\tilde{N}(ds,dh)$ is the compensated Poisson random measure.
But it is not clear whether a similar representation can be obtained for general L\'evy processes, or not, because it is very difficult to find
the fundamental solution of the infinitesimal generator for L\'evy processes.

Salminen and Yor \cite{Sal} used the potential theoretic approach and constructed the Tanaka formula for a symmetric L\'evy process $X=(X_t)_{t\geq0}$, if the local time exists,
by using the continuous resolvent density $r_q$:
\[ h(X_t-x)-h(x)= \tilde{N}^x_t +L^x_t \]
where $h(x):=\lim_{q\downarrow0}(r_q(0)-r_q(x))$ which is called a renormalized zero resolvent, $\tilde{N}^x_t :=-\lim_{q\downarrow0}M^{q,x}_t$ is a martingale and $L^x_t$ is the local time at level $x$.
But the expression of the martingale part $\tilde{N}^x_t$ was not given.

In \cite{Yan1} and \cite{Yan2}, Yano obtained an invariant excessive function $h$ with respect to the killed process:
\[ \bE^0_x[h(X_t)]=h(x) \]
where $\bE^0_x$ is the expectation with respect to the law of a L\'evy process $X$ starting at $x$ killed upon hitting zero,
which associates with the Tanaka formula at level zero because the local time for such a process at level zero becomes zero.
In the symmetric case Yano \cite{Yan1} assumed a necessary and sufficient condition for the existence of local times, and Salminen and Yor \cite{Sal} also assumed the same condition,
but in the asymmetric case Yano \cite{Yan2} needed sufficient conditions for the existence of the function and its expression.
Our result also gives the existence and its expression in the asymmetric case under weaker conditions than the ones in \cite{Yan2}.
Our approach is different from Pant\'i \cite{Pan}.

In Section \ref{sec2}, we shall give the preliminaries about resolvent operators of L\'evy processes and a connection between the local time and the resolvent density.
The convergence and its expression of the renormalized zero resolvent are mentioned in Section \ref{sec3}.
In Section \ref{sec4}, the Doob--Meyer decomposition can be constructed in the case of asymmetric L\'evy processes.
And then, we obtain the Tanaka formula for asymmetric L\'evy processes and the invariant excessive function with respect to the killed process.
In Section \ref{sec5}, we give several examples that satisfy the conditions introduced in Section \ref{sec4}.

\section{Preliminaries}\label{sec2}
Let $\cS(\bR)$ be the Schwartz space of rapidly decreasing functions on $\bR$.
We denote the law of processes starting at $x$ and the corresponding expectation by $\bP_x$ and $\bE_x$ respectively.

Consider a L\'evy process $X=(X_t)_{t\geq0}$ on $\bR$ with the L\'evy--Khintchine representation given by
\begin{equation*}
\bE_0[e^{iuX_t}]=e^{t\eta(u)},
\end{equation*}
where the L\'evy symbol $\eta$ of $X$ can be represented as
\begin{equation*}
\eta(u)=ibu-\frac{1}{2}au^2+\int_{\bR \setminus \{0\}}\left( e^{iuy}-1-iuy1_{|y|\leq1} \right) \nu(dy)
\end{equation*}
for constants $b \in \bR$ and $a \geq 0$ and a L\'evy measure $\nu$ on $\bR \setminus \{0\}$ satisfying $\int_{\bR\setminus\{0\}}(|y|^2 \wedge 1)\nu(dy) < \infty$.
We note that the L\'evy symbol $\eta$ is continuous.
Let $\Re\eta$ and $\Im\eta$ be the real and imaginary parts of $\eta$ respectively.
Remark that $\Re\eta\leq0$,
$\Re\eta$ is even and $\Im\eta$ is odd.

Let $T_0$ be the first hitting time to $0$ of $X$:
\[ T_0 := \inf\{t>0 : X_t=0\}. \]
We say that $0$ is regular for itself if $\bP_0(T_0=0)=1$,
and irregular for itself otherwise.
From the Blumenthal zero-one law,
$0$ is irregular if $\bP_0(T_0=0)=0$.

We introduce the following conditions:
\begin{description}
\item[\textbf{(A1)}] The L\'evy symbol $\eta$ satisfies that
\[ \int_\bR \Re \left(\frac{1}{q-\eta(u)}\right)du<\infty, \quad \text{for all $q>0$},\]
\item[\textbf{(A2)}] $0$ is regular for itself.
\end{description}

Denote the resolvent operator of the process $X$ by
\begin{equation*}
R_qf(x):=\bE_x\left[\int^\infty_0 e^{-qt}f(X_t)dt\right], \quad q>0, x \in \bR
\end{equation*}
for all bounded Borel measurable function $f$.
Denote the Fourier transform of $f \in \cS(\bR)$ by
\begin{equation*}
\cF[f](u) := \int_\bR e^{-iux} f(x)dx, \quad u \in \bR,
\end{equation*}
and the inverse Fourier transform by
\begin{equation*}
\cF^{-1}[f](x) := \frac{1}{2\pi} \int_\bR e^{iux} f(u)du, \quad x \in \bR.
\end{equation*}
Then, the resolvent operator is also represented as follows.

\begin{proposition}[{\cite[Proposition I.9]{Bert}}]\label{prop1}
For any $f \in \cS(\bR)$ and $x \in \bR$,
\begin{equation*}
R_qf(x)=\cF^{-1}\left[\frac{1}{q-\eta(u)}\cF[f](u)\right](x), \quad q>0.
\end{equation*}
\end{proposition}

Denote the resolvent kernel by $R_q(x,dy)$ for all $x \in\bR$ such that
\begin{equation*}
R_qf(x)=\int_\bR f(y)R_q(x,dy)
\end{equation*}
for all bounded Borel measurable function $f$.
It is known that the condition \textbf{(A1)} is equivalent to the existence of its density.
See \cite{Bert, Bre, Kes}.
\begin{remark}
\rm{
In \cite[Theorem II.16]{Bert},
the condition \textbf{(A1)} holds if and only if the resolvent kernel $R_q(0,dy)$ is absolutely continuous with respect to the Lebesgue measure and has a bounded density $r_q$.
}
\end{remark}

It is known that the condition \textbf{(A2)} is equivalent to the continuity of its density.
See \cite{Bert, Blu, Bre, Get, Kes}.
\begin{lemma}[{\cite[Theorem II.19]{Bert}}]\label{lem1}
Suppose that the condition \textbf{(A1)} holds.
Then, the followings hold for all $q>0$:
\begin{enumerate}[$(i)$]
\item The condition \textbf{(A2)} holds if and only if there exist a bounded continuous resolvent density $r_q$ such that
\[ R_qf(x)=\int_\bR f(y)r_q(y-x)dy,\]
for all bounded Borel measurable function $f$ and that
\[ \bE_x[ e^{-qT_0}]= \frac{r_q(-x)}{r_q(0)}, \quad x\in\bR. \]
\item If $r_q$ is continuous, then
\[ r_q(0)=\frac{1}{\pi}\int^\infty_0\Re\left(\frac{1}{q-\eta(u)}\right)du, \]
and for all $x \in \bR$
\[ 2r_q(0)-\{r_q(x)+r_q(-x)\}=\frac{2}{\pi}\int^\infty_0 \Re\left(\frac{1-\cos (ux)}{q-\eta(u)}\right)du. \]
\end{enumerate}
\end{lemma}

We introduce the following conditions:
\begin{description}
\item[\textbf{(A3)}] The process $X$ is the type C, i.e.,
\[ \text{either} \quad a > 0 \quad \text{or} \quad \int_{|y|\leq1}|y|\nu(dy)=\infty, \]
\item[\textbf{(A4)}] The process $X$ is not a compound Poisson process.
\end{description}

The following was proved by Kesten \cite{Kes}, and another proof was given by Bretagnolle \cite{Bre}.
\begin{lemma}[\cite{Kes} and \cite{Bre}]\label{lem2}
The conditions \textbf{(A1)} and \textbf{(A3)} hold
if and only if the conditions \textbf{(A2)} and \textbf{(A4)}.
Furthermore, under the condition \textbf{(A1)}, the condition \textbf{(A2)} holds
if and only if the condition \textbf{(A3)} holds.
\end{lemma}

In order to construct the Tanaka formula via the techniques in the potential theory,
we use a connection between the local time and the resolvent density.
\begin{lemma}[{\cite[Lemma V.3]{Bert}}]\label{lem3}
Suppose that the conditions \textbf{(A1)} and \textbf{(A2)} hold.
For any $x \in \bR$, denote by $dL^x_t$ the Stieltjes measure of the increasing function $L^x_\cdot$.
Then, it holds that
\begin{equation*}
\bE_y\left[\int^\infty_0 e^{-qt} dL^x_t\right]=r_q(x-y), \quad q>0, y \in \bR.
\end{equation*}
\end{lemma}

\begin{remark}
\rm{
In \cite[Theorem V.1]{Bert}, the condition \textbf{(A1)} holds if and only if
the occupation measure $\mu_t$ satisfying
for each non-negative Borel measurable function $f$ and $t \geq 0$,
\begin{equation*}
\int^t_0 f(X_s) ds = \int_\bR f(x) \mu_t(dx),
\end{equation*}
has the density in $L^2(dx\otimes d\bP_0)$ as the Radon--Nikodym derivative.
Therefore, if the condition \textbf{(A1)} holds, the local time for the process $X$ exists.
If the condition \textbf{(A1)} holds, then under the condition \textbf{(A2)} the local time $L^x_t$ is continuous almost surely with respect to $t$.
In the symmetric case,
if the condition \textbf{(A1)} holds, then the condition \textbf{(A2)} holds.
}
\end{remark}

\begin{remark}
\rm{
By Blumenthal and Getoor \cite{Blu},
it can be considered as the potential theoretic definition of the local time,
i.e. the local time can be defined as
a positive additive functional $L_t^x$ such that
\begin{equation*}
\bE_0\left[\int^\infty_0 e^{-qt} dL^x_t\right]=r_q(x).
\end{equation*}
}
\end{remark}

\section{Renormalized zero resolvent}\label{sec3}
Using the Fourier transform for $L^2(\bR)$-functions,
the resolvent density can be represented as follows:

\begin{proposition}\label{prop2}
Suppose that the conditions \textbf{(A1)} and \textbf{(A2)} hold.
The bounded continuous resolvent density can be expressed as:
\begin{equation*}
r_q(x)=\cF^{-1}\left[ \frac{1}{q-\eta(u)} \right](-x)
\end{equation*}
for all $q>0$ and $x \in \bR$.
\end{proposition}

\begin{proof}
Since $\Re(q-\eta(u)) \geq q$,
we have
\begin{align*}
\left|\frac{1}{q-\eta(u)}\right|^2&\leq \frac{\Re(q-\eta(u))}{q|q-\eta(u)|^2}
=\frac{1}{q}\Re\left(\frac{1}{q-\eta(u)}\right).
\end{align*}
Thus, by the condition \textbf{(A1)} we have $1\slash (q-\eta(u)) \in L^2(\bR)$.
By Proposition \ref{prop1} and Parseval's theorem,
we have for all $\phi \in \cS(\bR)$,
\begin{align*}
R_q\phi(x) &= \cF^{-1}\left[ \frac{1}{q-\eta(u)}\cF[\phi](u)\right](x)
\\&=\frac{1}{2\pi}\int_\bR \frac{e^{iux}}{q-\eta(u)}\cF[\phi](u)du
\\&=\frac{1}{2\pi}\int_\bR\cF\left[ \frac{e^{iux}}{q-\eta(u)} \right](y)\phi(y)dy
\\&=\int_\bR \cF^{-1}\left[ \frac{1}{q-\eta(u)}\right](x-y)
\phi(y)dy.
\end{align*}
From the definition of the resolvent operator $R_q$,
we then have for all $\phi \in \cS(\bR)$,
\begin{align*}
\int_\bR \left(r_q(y)-\cF^{-1}\left[ \frac{1}{q-\eta(u)} \right](-y)\right)\phi(y)dy=0.
\end{align*}
Since $r_q$ is continuous and integrable, by Lemma \ref{lem1}(i),
we have
\begin{equation*}
r_q(x)=\cF^{-1}\left[ \frac{1}{q-\eta(u)} \right](-x)
\end{equation*}
for all $q>0$ and $x \in \bR$.
\end{proof}

We introduce the following condition:
\begin{description}
\item[\textbf{(A)}] The L\'evy symbol $\eta$ satisfies that
\[ \frac{1}{q-\eta(u)} \in L^1(\bR), \quad \text{for all $q >0$}. \]
\end{description}

\begin{corollary}\label{cor1}
Suppose that the condition \textbf{(A)} holds.
The bounded continuous resolvent density $r_q$ can be expressed as:
\begin{equation*}
r_q(x)=\frac{1}{\pi}\int^\infty_0 \Re\left( \frac{e^{-iux}}{q-\eta(u)} \right) du
\end{equation*}
for all $q>0$ and $x \in \bR$.
\end{corollary}

From Lemma \ref{lem1}(i), We have the following:
\begin{corollary}\label{cor2}
If the condition \textbf{(A)} holds, then the conditions \textbf{(A1)} and \textbf{(A2)} hold.
\end{corollary}

\begin{remark}
\rm{
From Lemma \ref{lem2}, we know that if the condition \textbf{(A)} holds,
then the conditions \textbf{(A1)}, \textbf{(A2)}, \textbf{(A3)} and \textbf{(A4)} hold.
}
\end{remark}

\begin{remark}
\rm{
An asymmetric Cauchy process ($\al =1, \be\neq0$) does not satisfy the condition \textbf{(A)} but satisfy the conditions \textbf{(A1)} and \textbf{(A2)}.
}
\end{remark}

Now, we set
\begin{equation*}
h_q(x):=r_q(0)-r_q(-x), \quad q>0, x \in \bR.
\end{equation*}
From Lemma \ref{lem1}(i),
since $0\leq r_q(y)\leq r_q(0)$ for all $y \in \bR$,
then we have $h_q \geq 0$.
In \cite{Yan2},
the limit $h := \lim_{q \downarrow 0} h_q$ is called the renormalized zero resolvent if the limit exists,
which is known as a harmonic function for the killed process under some conditions.

But its convergence of $h_q$ is not clear for the asymmetric case, and Yano \cite{Yan2} needed the following conditions:
\begin{description}
\item[\textbf{(L1)}] The L\'evy symbol $\eta$ satisfies that
\[ \int^\infty_0 \frac{1}{q-\Re\eta(u)}du < \infty, \quad \text{for all $q >0$}, \]
\item[\textbf{(L2)}] The process $X$ is the type C, that is the same condition as \textbf{(A3)},
\item[\textbf{(L3)}] The real and imaginary parts of the L\'evy symbol $\eta$ have measurable derivatives on $(0,\infty)$ which satisfy
\[ \int^\infty_0 (u^2\wedge 1)\frac{|\Re\eta(u)'|+|\Im\eta(u)'|}{\Re\eta(u)^2+\Im\eta(u)^2}du < \infty. \]
\end{description}

However, we suppose the condition \textbf{(A)}, which is weaker than the condition \textbf{(L1)}.
The condition \textbf{(L2)} holds under the condition \textbf{(A)}.
Moreover, we shall introduce the condition \textbf{(B)} which is weaker than the condition \textbf{(L3)}:
\begin{description}
\item[\textbf{(B)}] The L\'evy symbol $\eta$ satisfies that
\[ \int^1_0\left|\Im\left(\frac{u}{\eta(u)}\right)\right| du <\infty. \]
\end{description}

\begin{theorem}\label{thm1}
Suppose that the condition \textbf{(A)} and \textbf{(B)} hold.
For all $x\in\bR$,
\begin{align*}
\lim_{q\downarrow 0} h_q(x)
=\frac{1}{\pi}\int^\infty_0 \Re\left(\frac{e^{iux}-1}{\eta(u)}\right)du =:h(x).
\end{align*}
\end{theorem}

To show Theorem \ref{thm1} and establish the Tanaka formula,
we need the following lemma.
\begin{lemma}\label{lem4}
Suppose that the condition \textbf{(A)} holds.
Then, the followings hold:
\begin{flalign*}
(i)\quad&|\eta(u)| \to \infty \quad \text{as $|u| \to \infty$}.&
\\(ii)\quad&\int^\infty_c \left| \frac{1}{\eta(u)}\right| du < \infty 
\quad \text{for all $c>0$}.&
\\(iii)\quad&\int^c_0 \left| \frac{u^2}{\eta(u)}\right|du < \infty
\quad \text{for all $c>0$}.&
\\(iv)\quad&\lim_{q\downarrow 0}\int_\bR\left|\frac{q}{q-\eta(u)}\right|du = 0.&
\end{flalign*}
\end{lemma}

\begin{proof}
(i) Since $r_1 \in L^1(\bR)$, $\cF[r_1](u)=1\slash(1-\eta(-u))$ and
\begin{equation*}
\left| \frac{1}{1-\eta(-u)} \right| \geq \frac{1}{1+|\eta(-u)|},
\end{equation*}
then, by the Riemann--Lebesgue theorem,
we have $|\eta(u)| \to \infty$ as $|u| \to \infty$. 

(ii) By Corollary \ref{cor2} and Lemma \ref{lem2}, the condition \textbf{(A3)} holds.
We then know $\Re\eta(u) \neq 0$ for $u \neq 0$.
By the condition \textbf{(A)}, we have
\begin{equation*}
\int_\bR \left|\frac{1}{1-\eta(u)}\right|du < \infty.
\end{equation*}
By the assertion (i), we know $|\eta(u)\slash(1-\eta(u))| \to 1$ as $|u| \to \infty$.
Hence, the required result follows.

(iii) Since we know $1-\cos(x) \geq x^2 \slash 4$ for $|x| \leq1$, by the condition \textbf{(A3)},
we have for all $0 < u \leq 1$
\begin{align*}
\left| \frac{\eta(u)}{u^2}\right| &\geq -\frac{\Re\eta(u)}{u^2}&
\\ &\geq \frac{a}{2}+\int_{|y|\leq |u|^{-1}}\frac{1-\cos(uy)}{(uy)^2}y^2\nu(dy)&
\\ &\geq \frac{a}{2}+\frac{1}{4}\int_{|y|\leq |u|^{-1}}y^2\nu(dy)&
\\ &\geq \frac{a}{2}+\frac{1}{4}\int_{|y|\leq1}y^2\nu(dy)>0.&
\end{align*}
Hence, the required result follows from the dominated convergence theorem.

(iv) For each $q<1$, we have $|q\slash(q-\eta(u))| \leq 1 \wedge |1\slash\eta(u)|$.
Thus, by the assertion (ii) and the dominated convergence theorem, we have
\begin{align*}
\lim_{q\downarrow 0}\int_\bR\left|\frac{q}{q-\eta(u)}\right|du 
&=\int_\bR \lim_{q\downarrow 0} \left|\frac{q}{q-\eta(u)}\right| du
\\&= 0. \qedhere
\end{align*}
\end{proof}

Now, we shall prove Theorem \ref{thm1}.
\begin{proof}[Proof of Theorem \ref{thm1}]
By Corollary \ref{cor1}, we have for each $x \in \bR$,
\begin{align*}
h_q(x)&=\frac{1}{\pi}\int^\infty_0 \Re\left(\frac{1-e^{iux}}{q-\eta(u)}\right)du
\\&=\frac{1}{\pi}\int^\infty_0 \Re\left(\frac{1-\cos(ux)}{q-\eta(u)}\right)du
+\frac{1}{\pi}\int^\infty_0 \Im\left(\frac{\sin(ux)}{q-\eta(u)}\right)du.
\end{align*}
Since we have for all $u\in\bR$,
\[ \left|\Re\left(\frac{1-\cos(u)}{q-\eta(u)}\right)\right|
\leq \frac{u^2 \wedge 1}{|\eta(u)|}, \]
by Lemma \ref{lem4}(ii), (iii) and using the dominated convergence theorem,
we have
\begin{equation*}
\int^\infty_0 \Re\left(\frac{1-\cos(u)}{q-\eta(u)}\right)du
\to \int^\infty_0 \Re\left(\frac{\cos(u)-1}{\eta(u)}\right)du,
\end{equation*}
as $q \downarrow 0$. 
Since we have
\begin{align*}
\left|\Im\left(\frac{\sin(u)}{q-\eta(u)}\right)\right|
\leq \left|\Im\left(\frac{u\wedge1}{\eta(u)}\right)\right|
\leq \left|\Im\left(\frac{u}{\eta(u)}\right)\right| \wedge \left|\frac{1}{\eta(u)}\right|,
\end{align*}
by the condition \textbf{(B)}, Lemma \ref{lem4}(ii) and using the dominated convergence theorem,
we have
\begin{equation*}
\int^\infty_0 \Im\left(\frac{\sin(ux)}{q-\eta(u)}\right)du
\to -\int^\infty_0 \Im\left(\frac{\sin(ux)}{\eta(u)}\right)du,
\end{equation*}
as $q \downarrow 0$. 
\end{proof}

\section{Tanaka formula}\label{sec4}
Using  Lemma \ref{lem3},
we can construct the Doob--Meyer decomposition as stated in \cite[Proposition 1]{Sal}.
\begin{proposition}\label{prop3}
Suppose that the conditions \textbf{(A1)} and \textbf{(A2)} hold.
For each $q>0$, $t\geqslant0$ and $x \in \bR$, it holds that
\begin{equation*}
r_q(-X_t+x)=r_q(-X_0+x)+M^{q,x}_t+q\int^t_0r_q(-X_s+x)ds-L^x_t,
\end{equation*}
where $M^{q,x}_t$ is a martingale with respect to the natural filtration $\{\cG_t\}_{t\geq0}$ of $X$.
\end{proposition}

\begin{proof}
By Lemma \ref{lem3} and the Markov property,
we have
\begin{align}\label{prop3-1}
\bE_{X_0}\left[ \int^\infty_0 e^{-qu}dL^x_u | \cG_s\right]
&=\int^s_0e^{-qu}dL^x_u + \bE_{X_s}\left[ \int^\infty_0 e^{-q(s+u)}dL^x_u\right] \notag
\\&=\int^s_0e^{-qu}dL^x_u + e^{-qs}r_q(-X_s+x).
\end{align}
Using the integration by parts, and by \eqref{prop3-1}, we obtain
\begin{align}\label{prop3-2}
&q\int^t_0e^{qs}\int^s_0 e^{-qu}dL^x_u ds \notag
\\&=e^{qt}\int^t_0 e^{-qu}dL^x_u - L^x_t \notag
\\&=e^{qt}\left(\bE_{X_0}\left[ \int^\infty_0 e^{-qu}dL^x_u | \cG_t\right]
 -e^{-qt}r_q(-X_t+x) \right) - L^x_t \notag
\\&=e^{qt}\bE_{X_0}\left[ \int^\infty_0 e^{-qu}dL^x_u | \cG_t\right]
 -r_q(-X_t+x) - L^x_t
\end{align}
Hence, by \eqref{prop3-1} and \eqref{prop3-2} we have
\begin{align}\label{prop3-3}
&r_q(-X_t+x)-q\int^t_0r_q(-X_s+x)ds+L^x_t \notag
\\&=-q\int^t_0e^{qs}\bE_{X_0}\left[\int^\infty_0 e^{-qu}dL^x_u | \cG_s \right]ds
 +e^{qt}\bE_{X_0}\left[ \int^\infty_0 e^{-qu}dL^x_u | \cG_t\right]
\end{align}
For the sake of simplicity of notations,
we shall write
\begin{align*}
Y_t&:=\bE_{X_0}\left[ \int^\infty_0 e^{-qu}dL^x_u | \cG_t\right],
\\ Z_t &:=-q\int^t_0e^{qs}Y_s ds+e^{qt}Y_t.
\end{align*}
Since we know $Z_0=r_q(-X_0+x)$,
by \eqref{prop3-3},
we will show that $Z_t$ is a martingale with respect to the natural filtration $\{\cG_t\}_{t\geq0}$.
By Fubini's theorem, we have for all $0 \leq v <t$,
\begin{align*}
\bE_{X_0}[Z_t | \cG_v]
&=-q\int^t_0 e^{qs}\bE_{X_0}[Y_s|\cG_v]ds+e^{qt}\bE_{X_0}[Y_t|\cG_v]
\\&=-q\int^v_0 e^{qs}Y_s ds-q\int^t_v e^{qs}Y_v ds+e^{qt}Y_v
\\&=-q\int^v_0 e^{qs}Y_s ds+e^{qv}Y_v
\\&=Z_v,
\end{align*}
and the required result follows. 
\end{proof}

Now we will establish the Tanaka formula for asymmetric L\'evy processes.
\begin{theorem}\label{thm2}
Suppose that the conditions \textbf{(A)} and \textbf{(B)} hold.
Let $h$ and $M^{q,x}$ be the same as in Theorem \ref{thm1} and Proposition \ref{prop3} respectively. Then, for each $t\geq0$ and $x \in \bR$,
it holds that
\begin{equation*}
h(X_t-x)=h(X_0-x)+\tilde{N}^x_t + L^x_t,
\end{equation*}
where $\tilde{N}^x_t := - \lim_{q\downarrow0}M^{q,x}_t$ is a martingale.
\end{theorem}

\begin{proof}
From the Doob--Meyer decomposition (Proposition \ref{prop3}),
let $q \downarrow 0$ and by Theorem \ref{thm1},
then we have
\begin{equation*}
h(X_t-x)=h(X_0-x)
-\lim_{q \downarrow 0}\left(M^{q,x}_t+q\int^t_0 r_q(-X_s+x)ds\right)
+L^x_t.
\end{equation*}
Recall that $0\leq r_q(y)\leq r_q(0)$ for all $y \in \bR$,
and then,
\begin{equation*}
0 \leq q\int^t_0 r_q(-X_s+x)ds \leq qr_q(0)t.
\end{equation*}
Hence, by Lemma \ref{lem4}(iv),
\begin{equation}\label{thm2-1}
q\int^t_0 r_q(-X_s+x)ds \to 0 \quad \text{as $q \downarrow 0$}.
\end{equation}
It remains to show that $\tilde{N} := -\lim_{q\downarrow0}M^{q,x}$ is a martingale.
Thus, we will prove that
\begin{equation*}
\bE_0|\tilde{N}^x_t-M^{q,x}_t| \to 0 \quad \text{as $q \downarrow 0$}.
\end{equation*}
We know that
\begin{align*}
|\tilde{N}^x_t-M^{q,x}_t| &\leq |h(X_t-x)-h_q(X_t-x)| +|h(X_0-x)-h_q(X_0-x)|
\\&\quad +q\int^t_0r_q(-X_s+x)ds.
\end{align*}
By Theorem \ref{thm1}, the second term of the above right-hand side goes to $0$ as $q \downarrow 0$.
By \eqref{thm2-1},
the last term converges to $0$ as $q \downarrow 0$.
It remains to prove the convergence of the first term as $q \downarrow 0$.
Thus, it is enough to show that $h_q(X_t-x)$ converges in $L^1(d\bP_0)$ to $h(X_t-x)$ as $q \downarrow 0$.
Since $h_q(y)\geq0$ for any $y \in \bR$, we have
\begin{align*}
h_q(x) &\leq h_q(x) + h_q(-x)
\\&=\frac{2}{\pi}\int^\infty_0 \Re\left( \frac{1-\cos(ux)}{q-\eta(u)} \right)du
\\&\leq \frac{2}{\pi}\int^\infty_0 \frac{1-\cos(ux)}{|\eta(u)|}du.
\end{align*}
Using Fubini's theorem, Lemma \ref{lem4}(ii) and (iii),
we have
\begin{align*}
&\bE_0\left[ \int^\infty_0 \frac{1-\cos(u(X_t-x))}{|\eta(u)|}du\right]
\\&=\int^\infty_0 \frac{1-\Re\exp \left(t\eta(u)-iux\right)}{|\eta(u)|} du
\\&=\int^\infty_0 \frac{1-\cos(t\Im\eta(u)-ux)\exp\left(t\Re\eta(u)\right)}{|\eta(u)|} du
\\&\leq \int^1_0 \frac{1-\cos(t\Im\eta(u)-ux)-t\Re\eta(u)}{|\eta(u)|} du
+\int^\infty_1 \left|\frac{2}{\eta(u)}\right|du
\\&\leq \int^1_0 \frac{\left(t\Im\eta(u)-ux\right)^2}{|\eta(u)|} du
 +\int^\infty_1 \left|\frac{2}{\eta(u)}\right|du + t
\\&\leq 2\int^1_0 \frac{\left(t\Im\eta(u)\right)^2+ (ux)^2}{|\eta(u)|} du
 +\int^\infty_1 \left|\frac{2}{\eta(u)}\right|du + t <\infty.
\end{align*}
Hence, it follows from the dominated convergence theorem.
Therefore,
\begin{equation*}
\bE_0 |\tilde{N}^x_t-M^{q,x}_t| \to 0 \quad \text{as $q \downarrow 0$}.
\end{equation*}
The proof is now complete.
\end{proof}

\begin{remark}
\rm{
From Theorem \ref{thm2}, we obtain the invariant excessive function with respect to the killed process.
Indeed, when we denote the law of the process starting at $x$ killed upon hitting zero and the corresponding expectation by $\bP^0_x$ and $\bE^0_x$ respectively,
under the condition \textbf{(A)} and \textbf{(B)},
we have,
\begin{equation*}
\bE^0_x[h(X_t)]=h(x),
\end{equation*}
for all $t\geq0$ and $x \in \bR$.
}
\end{remark}

\section{Examples}\label{sec5}
We shall introduce examples satisfying the conditions \textbf{(A)} and \textbf{(B)}.
Because the condition \textbf{(A)} is a sufficient condition to have local times and explicit resolvent densities,
we give examples with a focus on satisfying the condition \textbf{(B)}.

\begin{example}[Stable process]\label{ex1}
\rm{
Let $X$ be an asymmetric stable process with index $\a l\in (1,2)$.
The L\'evy measure $\nu$ on $\bR \setminus \{0\}$ is given by
\begin{equation*}
\nu(dy) =
  \begin{cases}
  c_+|y|^{-\al-1}dy &\quad \text{on $(0,\infty)$}, \\
  c_-|y|^{-\al-1}dy &\quad \text{on $(-\infty,0)$},
  \end{cases}
\end{equation*}
where $\al \in (1,2)$, and $c_+$ and $c_-$ are non-negative constants such that $c_+ + c_- > 0$.
The L\'evy symbol $\eta$ of $X$ is represented as
\begin{equation*}
\eta(u) = -d|u|^\al \left(1-i\be\sgn(u)\tan\frac{\pi \al}{2}\right),
\end{equation*}
where $d > 0$ and $\be \in [-1,1]$ are given by
\begin{equation*}
d=\frac{c_+ +c_-}{2c(\al)}, \quad \be=\frac{c_+-c_-}{c_+ +c_-}
\end{equation*}
with
\begin{equation*}
c(\al)=\frac{1}{\pi}\Ga(\al+1)\sin\frac{\pi \al}{2}.
\end{equation*}
See \cite{Sat} on details.

Since we have for all $q>0$,
\begin{align*}
\left|\frac{1}{q-\eta(u)}\right|
\leq \frac{1}{q-\Re\eta(u)}
=\frac{1}{q+d|u|^{\al}},
\end{align*}
and $\al \in(1,2)$,
the process $X$ satisfies the condition \textbf{(A)}.

Since we have for all $0<u\leq1$,
\begin{align*}
\left|\Im\left(\frac{u}{\eta(u)}\right)\right|
\leq \left|\frac{u}{\eta(u)}\right|
\leq \frac{u}{|\Re\eta(u)|}
=\frac{1}{d}|u|^{1-\al},
\end{align*}
by $1-\al \in(-1,0)$,
the process $X$ satisfies the condition \textbf{(B)}.

In this case, it can be represented by
\begin{equation*}
h(x) = c(-\al) \frac{1 - \be \sgn(x)}
{d\left(1 + \be^2 \tan^2 (\pi \al \slash 2)\right)} |x|^{\al - 1}.
\end{equation*}
The result is consistent with \cite{Tsu}.
}
\end{example}

\begin{remark}
\rm{
In \cite{Tsu}, by using the Fourier transform,
we could find the fundamental solution $F$ of the infinitesimal generator for a stable process $S =(S_t)_{t\geq0}$ with index $\al \in(1,2)$.
Moreover, we have $F(x)=h(x)$ for all $x\in\bR$.
In \cite{Tsu}, since we used It\^o's stochastic calculus,
we have the martingale part $\tilde{N}^x_t$ of the Tanaka formula can be represented as the explicit form:
\[ \tilde{N}^x_t :=\int^t_0\int_{\bR\setminus\{0\}}\{F(S_{s-}-x+h)-F(S_{s-}-x)\}
\tilde{N}(ds,dh). \]
Thus, we could study the property of local times from the Tanaka formula.
On the other hand, for general L\'evy processes, even if the renormalized zero resolvent and the local time exist, we could not use It\^o's stochastic calculus, because we do not know the explicit form of the renormalized zero resolvent.
}
\end{remark}

\begin{example}[Truncated stable process]\label{ex2}
\rm{
A truncated stable process is a L\'evy process with the L\'evy measure $\nu$ on $\bR \setminus \{0\}$ given by
\begin{equation*}
\nu(dy) =
  \begin{cases}
  c_+|y|^{-\al -1}1_{\{y\leq1\}}dy &\quad \text{on $(0,\infty)$}, \\
  c_-|y|^{-\al -1}1_{\{y\geq-1\}}dy &\quad \text{on $(-\infty,0)$},
  \end{cases}
\end{equation*}
where $\al \in (1,2)$, and $c_+$ and $c_-$ are non-negative constants such that $c_+ + c_- > 0$.

Since we know $1-\cos(x) \geq x^2 \slash 4$ for $|x| \leq1$,
we have for all $u\geq1$,
\begin{align*}
-\Re\eta(u)&=\int_{\bR\setminus\{0\}} \left(1-\cos(uy)\right)\nu(dy)
\\&\geq \frac{1}{4}\int_{|y| \leq u^{-1}} (uy)^2\nu(dy)
\\&=\frac{c_++c_-}{4}\int^{u^{-1}}_0 u^2y^{-\al+1}dy
\\&=\frac{c_++c_-}{4(2-\al)}u^{\al}.
\end{align*}
We then have for all $q>0$,
\begin{align*}
\int^\infty_0 \left|\frac{1}{q-\eta(u)}\right|du
&\leq \int^\infty_0 \frac{1}{q-\Re\eta(u)}du
\\&\leq \frac{1}{q} + \frac{4(2-\al)}{c_++c_-}\int^\infty_1u^{-\al}du 
< \infty,
\end{align*}
by $-\al \in (-2,-1)$.
Thus, the process $X$ satisfies the condition \textbf{(A)}.

Since $|\sin(x)-x| \leq |x|^3$ for all $x\in\bR$,
we have for all $0 < u \leq1$,
\begin{align*}
\left| \frac{\Im\eta(u)}{u^3} \right|
&=\left| \int_{|y|\leq1}\frac{\sin(uy)-uy}{u^3}\nu(dy)\right|
\\&\leq \int_{|y|\leq1} \left| \frac{\sin(uy)-uy}{u^3}\right| \nu(dy)
\\&\leq \int_{|y|\leq1}|y|^3\nu(dy) < \infty.
\end{align*}
Thus, the process $X$ satisfies the condition \textbf{(B)}.
}
\end{example}

\begin{remark}
\rm{
If a L\'evy measure has a bounded support,
the condition \textbf{(B)} holds by the same argument as stated Example \ref{ex2}.
}
\end{remark}

\begin{example}[Tempered stable process]\label{ex3}
\rm{
A tempered stable process is a L\'evy process
with the L\'evy measure $\nu$ on $\bR \setminus \{0\}$ is given by
\begin{equation*}
\nu(dy) =
  \begin{cases}
  c_+|y|^{-\al_+ -1}e^{-\la_+ |y|}dy &\quad \text{on $(0,\infty)$}, \\
  c_-|y|^{-\al_- -1}e^{-\la_- |y|}dy &\quad \text{on $(-\infty,0)$},
  \end{cases}
\end{equation*}
where $\al_+, \al_- \in (1,2)$, and $c_+$, $c_-$, $\la_+$ and $\la_-$ are non-negative constants such that $c_+ + c_- > 0$.
The processes have studied as models for stock price behavior in finance.
See Carr et al. \cite{Car} on details.

Since we have for all $u\geq1$,
\begin{align*}
-\Re\eta(u)
&\geq \frac{1}{4}\int_{|y|\leq u^{-1}}(uy)^2\nu(dy)
\\&\geq \frac{u^2}{4}\left(c_+ e^{-\la_+}\int^{u^{-1}}_0 y^{-\al_+ +1}dy
+ c_- e^{-\la_-}\int^{u^{-1}}_0 y^{-\al_- +1}dy\right)
\\&=\frac{c_+ e^{-\la_+}}{4(2-\al_+)}u^{\al_+}
+ \frac{c_- e^{-\la_-}}{4(2-\al_-)}u^{\al_-},
\end{align*}
by $\al_+, \al_- \in (1,2)$,
the process $X$ satisfies the condition \textbf{(A)}.

In the case of $\la_+, \la_- >0$, or $c_+=0, \la_->0$, or $ c_-=0, \la_+>0$,
since $|\sin(x)-x| \leq |x|^3$ for all $x\in\bR$,
we have for all $0< u \leq1$,
\begin{align*}
\left| \frac{\Im\eta(u)}{u^3}\right|
\leq \int_{\bR \setminus \{0\}} \left| \frac{\sin(uy)-uy}{u^3} \right|\nu(dy)
\leq \int_{\bR \setminus \{0\}} |y|^3 \nu(dy) <\infty.
\end{align*}
Thus, this case satisfies the condition \textbf{(B)}.

In the case of $c_+>0, \la_+ =0$,
by the same argument as Example \ref{ex1},
we have for all $u \in \bR$,
\begin{align*}
-\Re\eta(u)
\geq \int^\infty_0 \left(1-\cos(uy)\right)\nu(dy)
=\frac{c_+}{2c(\al)}|u|^{\al_+}
\end{align*}
where $c(\al)=(1\slash\pi)\Ga(\al+1)\sin(\pi \al \slash 2)$.
Since we have for $0<u\leq1$,
\begin{align*}
\left|\Im\left(\frac{u}{\eta(u)}\right)\right|
=\left|\frac{u\Im\eta(u)}{(\Re\eta(u))^2+(\Im\eta(u))^2}\right|
\leq \frac{u}{2|\Re\eta(u)|}
\leq \frac{c(\al)}{c_+}u^{1-\al_+},
\end{align*}
by $1-\al_+\in (-1,0)$,
this case satisfies the condition \textbf{(B)}.

In the case of $c_->0, \la_-=0$,
the condition \textbf{(B)} holds by the same argument as the case of $\la_+=0$.

Thus, the process $X$ satisfies the condition \textbf{(B)}.
}
\end{example}

\begin{example}\label{ex4}
\rm{
Suppose that the condition \textbf{(A)} holds,
and that a L\'evy measure $\nu$ satisfies $\int_{|y|>1}|y|\nu(dy)<\infty$ and
$b \neq -\int_{|y|>1}y\nu(dy)$.

Since we have
\[ \left|\Im\left(\frac{u}{\eta(u)}\right)\right| \leq \left|\frac{u}{\Im\eta(u)}\right|, \]
and $|\sin(x)-x1_{|x|\leq1}|\leq|x|^3 \wedge |x|$ for all $x \in \bR$,
we have
\begin{align*}
\left|\frac{\Im\eta(u)}{u}\right|
&=\left|b + \int_{|y|\leq1}\frac{\sin(uy)-uy}{u}\nu(dy) +\int_{|y|>1}\frac{\sin(uy)}{u}\nu(dy)\right|
\\&\to \left| b+\int_{|y|>1}y\nu(dy)\right| >0,
\end{align*}
as $q \downarrow 0$.
By the dominated convergence theorem, the process $X$ satisfies the condition \textbf{(B)}.
}
\end{example}

\begin{example}[Spectrally positive or negative process]\label{ex5}
\rm{
A L\'evy process with no positive (negative) jumps is called a spectrally negative (positive) process. The processes have studied as models for insurance risk and dam theory.

Suppose that the condition \textbf{(A)} holds,
and that a L\'evy measure $\nu$ has a support in $(-\infty,0)$ and satisfies $\int_{|y|>1}|y|\nu(dy)<\infty$.
The processes are integrable spectrally negative processes satisfying the condition \textbf{(A)}.

In the case of $b \neq -\int_{|y|>1}y\nu(dy)$,
the process is one of Example \ref{ex4}.

We consider the case of $b = -\int_{|y|>1}y\nu(dy)$.
Since we have for all $x\in\bR$,
\begin{align*}
0 &\leq h_q(x)
\leq h_q(x)+h_q(-x)
\\&=\frac{2}{\pi}\int^\infty_0 \Re\left(\frac{1-\cos(ux)}{q-\eta(u)}\right)du.,
\end{align*}
by Lemma \ref{lem4}(ii) and (iii), we have
\begin{align*}
&\left|\int^1_0\Im\left(\frac{\sin(u)}{q-\eta(u)}\right)du\right|
\\&\leq \left|\int^\infty_0\Re\left(\frac{1-\cos(u)}{q-\eta(u)}\right)du\right|
+ \left|\int^\infty_1\Im\left(\frac{\sin(u)}{q-\eta(u)}\right)du\right|
\\&\leq \int^\infty_0\frac{|u|^2\wedge1}{|\eta(u)|}du
+ \int^\infty_1\left|\frac{1}{\eta(u)}\right|du <\infty.
\end{align*}
Since $\Im\eta(u) \geq 0$ for all $u \geq0$,
we have for all $0<u\leq1$,
\begin{align*}
\Im\left(\frac{\sin(u)}{q-\eta(u)}\right)
= \frac{\Im\eta(u)\sin(u)}{(q-\Re\eta(u))^2+(\Im\eta(u))^2}
\end{align*}
is increasing as $q \downarrow 0$,
by the monotone convergence theorem,
the condition \textbf{(B)} follows.

Integrable spectrally positive processes satisfying the condition \textbf{(A)}
also satisfy the condition \textbf{(B)}
by same argument as the spectrally negative case.
}
\end{example}

\begin{remark}
\rm{
Example \ref{ex1} also satisfies the condition \textbf{(L3)}.
But Example \ref{ex2} and \ref{ex3} do not satisfy the condition \textbf{(L3)},
and in Example \ref{ex4} and \ref{ex5} there exists processes that do not satisfy the condition \textbf{(L3)}.
}
\end{remark}

\section*{Acknowledgements}
I would like to thank Professor Atsushi Takeuchi of Osaka City University and Professor Kouji Yano of Kyoto University for their valuable advice.

\end{document}